\documentclass[reqno,12pt,a4paper]{amsart}

\usepackage{amsmath, amsthm, amsfonts, amssymb}

\usepackage[applemac]{inputenc}

\theoremstyle{plain}
\newtheorem{theorem}{Theorem}
\newtheorem{proposition}[theorem]{Proposition}
\newtheorem{lemma}[theorem]{Lemma}

\theoremstyle{definition}

\theoremstyle{remark}
\newtheorem{remark}[theorem]{Remark}

\DeclareMathOperator{\im}{Im}
\def\Z{\mathbb{Z}}	
	
\def\R{\mathbb{R}}	
\renewcommand{\leq}{\leqslant} 		
\renewcommand{\geq}{\geqslant}

\let\ker\relax
\DeclareMathOperator{\ker}{Ker}
\def\cA{\mathcal{A}}
\def\hcA{\hat{\cA}}
\def\cF{\mathcal{F}}
\def\hcF{\hat{\cF}}

\begin{document}

\title{Bihamiltonian cohomology of KdV brackets}

\author{Guido Carlet}

\author{Hessel Posthuma}

\author{Sergey Shadrin}

\address{Korteweg-de Vries Instituut voor Wiskunde, Universiteit van Amsterdam, Postbus 94248,
1090GE Amsterdam, Nederland}

\email{g.carlet@uva.nl, h.b.posthuma@uva.nl, s.shadrin@uva.nl}

\begin{abstract}
Using spectral sequences techniques we compute the bihamiltonian cohomology groups of the pencil of Poisson brackets of dispersionless KdV hierarchy. In particular this proves a conjecture of Liu and Zhang about the vanishing of such cohomology groups. 
\end{abstract}

\maketitle 

\tableofcontents
 
\section{Introduction}

A class of important examples of integrable systems comes from the study of the Poisson structures on a formal loop space, that is, on the space of maps from the circle to an $n$-dimensional ball considered in a formal algebraic way. In this setting one can generalize various notions of differential geometry, in particular one can define the space of functional polyvector fields equipped with the Schouten bracket. 
Introducing a Poisson structure and a system of Hamiltonians in involution one obtains a hierarchy of commuting PDEs. The equations of the flows are given by differential polynomials in involving the derivatives with respect to the spatial variable. 
In the case of the so-called bihamiltonian systems, which are characterized by the presence of two compatible Poisson structures, one can construct a set of Hamiltonians in involution with respect two both Poisson structures, starting from the Casimirs of one of them, by the so called Lenard-Magri recursion scheme. Such framework is especially relevant in the applications of the theory of integrable systems to the Gromov-Witten theory. For further details see~\cite{dz01}. 

In this context the theory of deformations of a compatible pair of Poisson structures is very important. In particular we want to consider the formal deformation of a pair of compatible brackets up to the action of the so-called Miura group, which is the group of local diffeomorphisms of the space of dependent variables. We assume that these diffeomorphisms may depend on the formal deformation parameter. 

An algebraic reformulation of this deformation problem leads to the concept of bihamiltonian cohomology of a pair of Poisson structures. As usual in the deformation theory of algebraic structures, we can classify, in terms of some cohomology groups, the space that parametrizes possible infinitesimal deformations of a pair of brackets and the space that parametrizes the obstruction for a formal perturbative integration of an infinitesimal deformation. 

The computation of
these cohomology groups in general case is fairly difficult, and we refer
the reader to the papers~\cite{al11,b08,dlz06,lz05,lz11,lz13,l02} for various particular
computations as well as surveys and reformulations of these ideas.

In a recent paper~\cite{lz13}, Liu and Zhang computed the bihamiltonian cohomology for the pair of brackets of the dispersionless KdV hierarchy up to degree $3$ and conjectured that the cohomology vanishes in higher degree. 

In this paper we propose a new method for the computation of the bihamiltonian cohomology groups. We introduce a filtration of a related polynomial complex and show that the corresponding spectral sequence collapses at the second page. That allows us to reconstruct the cohomology of such complex and consequently to obtain the bihamiltonian cohomology of the dispersionless KdV Poisson pencil. 
In particular we reproduce the results of Liu and Zhang and prove their conjecture on the vanishing of the cohomology in higher degrees. 

\subsection{Outline of the paper}

In Section 2 we recall the formalism of local functional polyvector fields and the definition of bihamiltonian cohomology groups.  
In Section 3 we recall some textbook material on spectral sequences that we use to approach the problem. 
In Section 4 we introduce the filtration of the polynomial complex, we prove our theorems on its cohomology using the induced spectral sequence and derive the bihamiltonian cohomology of the dispersionless KdV Poisson pencil.

\section{Formalism of functional polyvector fields}

In this Section, we introduce a minimal version of the necessary formalism of local functional polyvector fields and pose the main problem about the bihamiltonian cohomology for the dispersionless KdV hierarchy. For further details in the general case we refer to~\cite{lz13,dz01}.

\subsection{Basic definitions}

Consider a supercommutative associative algebra $\hcA$ defined as  
\begin{equation}
\hcA := C^\infty(\R) [[u^1, u^2, \dots; \theta^0, \theta^1, \theta^2, \dots ]].
\end{equation}
Here $u^i$, $i=1,2,\dots$, are some formal even variables and $\theta^i$, $i=0,1,2,\dots$, are some odd variables.  An element in $C^\infty (\R)$ is represented by a function of the coordinate $u=u^0$ on $\R$. 

This algebra has two different gradations. We define the \emph{standard gradation} on $\hcA$ by assigning the degrees 
\begin{equation}
\deg u^i = \deg \theta^i = i, \quad i=0,1,2,\dots.
\end{equation}
and degree zero to both $\theta=\theta^0$ and the elements in $C^\infty (\R)$. The standard degree $d$ component of $\hcA$ is denoted $\hcA_d$. The {\it super gradation} is defined by assigning degree one to $\theta^i$ for $i\geq0$ and degree zero to the remaining generators of $\hcA$. The super degree $p$ component is denoted $\hcA^p$. We also denote 
\begin{equation}
\hcA_d^p=\hcA_d \cap \hcA^p .
\end{equation}

We define the standard derivation on $\hcA$
\begin{equation}
\partial = \sum_{s\geq0} \left( u^{s+1} \frac{\partial }{\partial u^s} + \theta^{s+1} \frac{\partial }{\partial \theta^s} \right)
\end{equation}
which is compatible with both gradations on $\hcA$, in particular it increases the standard degree by one and leaves invariant the super degree. 

The \emph{local functional polyvector fields} are defined as the elements of the quotient
\begin{equation}
\hcF = \hcA / \partial \hcA ,
\end{equation}
and the corresponding projection is denoted with the integral symbol
\begin{equation}
\int \colon \hcA \to \hcF .
\end{equation}
The gradations on $\hcA$, being compatible with the derivation $\partial$, define corresponding gradations on $\hcF$. The corresponding homogeneous components are denoted with the obvious notations $\hcF_d$, $\hcF^p$ and $\hcF_d^p$.
 
The Schouten bracket $[\cdot,\cdot]$ on $\hcA$ and, therefore, on $\hcF$ (abusing notation we denote it by the same symbol) is defined in the standard way using the second order operator 
$
\delta_u \delta_\theta,
$
where 
\begin{equation}
\delta_u :=\sum_{s=0}^\infty (-\partial)^s \frac{\partial}{\partial u^s}, 
\qquad
\delta_\theta :=\sum_{s=0}^\infty (-\partial)^s \frac{\partial}{\partial \theta^s}. 
\end{equation}

\subsection{Poisson structures}

A Poisson structure is given by a local functional bivector field $P\in \hcF^2$ such that $[P,P]=0$. The operator $d_P:=[P,\cdot]$ on $\hcF$ is induced from the operator $D_P$ defined on $\hcA$ given by the 
following formula
\begin{equation}
D_P:= \sum_{s=0}^\infty  \left(
\partial^s (\delta_\theta P) \frac{\partial}{\partial u^s}
+
\partial^s (\delta_u P) \frac{\partial}{\partial \theta^s}
\right) .
\end{equation}
 The deformation theory of a single Poisson bracket is controlled by the cohomology $H(\hcF,d_P)$. An important related space is $H(\hcA,D_P)$. 
 
Given $P \in \hcF^2_1$ a scalar Hamiltonian structure of hydrodynamic type, we have \cite{g02, dms05, lz11, lz13}
\begin{equation}
H(\hcF, d_P) = H(\hcA, D_P) = \R \oplus \R \theta .
\end{equation}
%
The cohomology $H(\hcA,D_P)$ was computed in~\cite{lz11}. 
In~\cite{lz13} the cohomology $H_{>0}(\hcF,d_P)$ is derived from $H(\hcA, D_P)$ by a long exact sequence argument. The cohomology $H_0(\hcF,d_P)$ can be easily computed in the scalar case by hand. 

A pencil of two commuting Poisson structures is given by two local functional bivectors $P_1,P_2\in \hcF^2$ such that $[P_1,P_1]=[P_2,P_2]=[P_1,P_2]=0$. In this case the deformation theory of the pair is controlled by the graded vector space 
\begin{equation}
BH(\hcF,d_{P_1},d_{P_2}):=\frac{\ker d_{P_1} \cap \ker d_{P_2}  }{  \im \left(d_{P_1} d_{P_2}\right) }
\end{equation}
 that is called the \emph{bihamiltonian cohomology}. For almost all values of the bi-grading $(p,d)$ the computation of the bihamiltonian cohomology can be reduced to the computation of 
\begin{equation}
BH(\hcA,D_{P_1},D_{P_2}):=\frac{ \ker D_{P_1} \cap \ker D_{P_2}  }{ \im \left(D_{P_1} D_{P_2}\right)}
\end{equation}
 via a long exact sequence~\cite{lz13}.

\subsection{Poisson structures of the dispersionless KdV} 

The two Poisson structures of the dispersionless KdV are given by the following two bi-vector fields:
\begin{align}
P_1 & := \frac{1}{2} \int \theta\theta^1,
&
P_2 & := \frac 12  \int u \theta \theta^1.
\end{align} 

The Schouten brackets with $P_1$ and $P_2$ induce the differentials $D_1$ and $D_2$ on the space $\hcA$ and the differentials $d_1$ and $d_2$ on the space $\hcF$. A direct computation shows that 
\begin{align}
D_1 & = \sum_{s\geq0} \theta^{s+1} \frac{\partial }{\partial u^s},
\notag \\
D_2 & = \sum_{s\geq0} \left(  \partial^s ( u \theta^1 + \frac12 u^1 \theta ) \frac{\partial }{\partial u^s} + \partial^s (\frac12 \theta \theta^1 ) \frac{\partial }{\partial \theta^s} \right). \notag
\end{align}

The main goal of this paper is to compute the graded vector space $BH(\hcF,d_{1},d_{2})$. A closely related space is $BH(\hcA,D_{1},D_{2})$ and we explain below the precise relation following the results of Liu and Zhang~\cite{lz13}.

\subsection{Grading and subcomplexes}

Since $\theta^i$, $i=0,1,2,\dots$, are odd variables, we have a restriction on possible gradings $p$ and $d$ on the spaces $\hcA$ and $\hcF$. Indeed, the minimal possible  standard $d$-degree of a monomial in $\theta^i$ of super degree $p$ is the degree of $\theta^0\cdots\theta^{p-1}$ equal to $0+\cdots+(p-1)=p(p-1)/2$. So, $\hcA^p_d = \hcF^p_d=0$ for $d<p(p-1)/2$. 

Another observation is related to the fact that the $(p,d)$-degree of the operators $D_1$ and $D_2$ (and, therefore, $d_1$ and $d_2$) is $(1,1)$. Therefore, the difference $p-d$ is preserved by both operators, and this means that the whole space $\hcA$ (and $\hcF$) can be split into an a completion of an infinite direct sum of subcomplexes with the fixed difference $p-d$. The inequality $d<p(p-1)/2$ implies that each of these subcomplexes is finite. Let us give here first few examples (we write it for $\hcF$; for $\hcA$ it is exactly the same): 
\begin{align}
& 0 \to \hcF^1_0 \to \hcF^2_1 \to 0 \\ \notag
& 0 \to \hcF^0_0 \to \hcF^1_1 \to \hcF^2_2 \to \hcF^3_3 \to 0 \\ \notag
& 0 \to \hcF^0_1 \to \hcF^1_2 \to \hcF^2_3 \to \hcF^3_4 \to 0 \\ \notag
& 0 \to \hcF^0_2 \to \hcF^1_3 \to \hcF^2_4 \to \hcF^3_5 \to \hcF^4_6 \to 0 \\ \notag
& \vdots
\end{align}
So, we can compute the bihamiltonian cohomology of each subcomplex separately. Subcomplexes are fixed by the difference $d-p=-1,0,1,2\dots$.







\section{Spectral sequences}

In this Section we recall some basic machinery from the theory of spectral sequences. Although this material is well-known (see e.~g.~\cite{m01}), we have decided to include it for the benefit of the reader, and because the precise construction of the pages of the spectral sequence allows in the end to reconstruct the cocycles representing nontrivial cohomology classes. In this way, our arguments in the next section computing the bihamiltonian cohomology using the spectral sequence of a filtration can be related to the explicit computations of [9].

A (cohomological type) {\it spectral sequence} is a family of differential $\Z$-bigraded vector spaces $(E_r^{*,*},d_r)$ with differentials $d_r$ of bidegree $(r,1-r)$ such that for all $p,q\in\Z$ and all $r\geq 0$, the page $E_{r+1}^{pq}$ is isomorphic to the cohomology of the previous page
\begin{equation}
H^{pq}(E_r^{*,*}, d_r) := \frac{\ker ( d_r: E_r^{pq} \to E_r^{p+r,q-r+1})}{\im (d_r:E_r^{p-r, q+r-1} \to E_r^{pq})} .
\end{equation}

A spectral sequence is said to {\it collapse} at the $N$-th term if the differentials $d_r$ vanish for $r\geq N$. In such case $E_r^{*,*} \cong E_N^{*,*}$ for $r\geq N$. Then $E_\infty^{*,*} := E_N^{*,*}$ is called the {\it limit term} of the spectral sequence.    

Let $(C,d)$ be a $\Z$-graded differential complex (of $\R$-vector spaces) with a differential $d:C\to C$, $d^2=0$, of degree $1$. 
Let $F^*C$ be a decreasing filtration of $(C,d)$, i.e. a family of vector subspaces $F^i C$ of $C$, $i\in\Z$
\begin{equation}
\cdots \subset F^{i+1} \subset F^i C \subset \cdots \subset C 
\end{equation}
which are compatible with the differential, $d( F^i C ) \subset F^i C$. 
Let $F^i C^p := F^i C \cap C^p$ be the induced filtration of the homogeneous component $C^p$ of degree $p$. 
Let us denote by $gr^* C$ and $gr^* C^p$ the graded spaces associated to the filtrations $F^* C$ and $F^* C^p$, respectively, i.e. 
\begin{equation}
gr^i C = \frac{F^i C}{F^{i+1} C}  , \quad
gr^i C^p = \frac{F^i C^p}{F^{i+1} C^p} .
\end{equation}

\begin{proposition}
With a $\Z$-graded differential complex $(C,d)$ with a filtration $F^* C$ one associates a spectral sequence $(E^{*,*}_r,d_r)$. The first two pages are given by
\begin{equation} \label{E0}
E_0^{p,q} = gr^p C^{p+q} 
\end{equation}
and
\begin{equation} \label{E1}
E_1^{p,q} =  \frac{d^{-1} (F^{p+1} C^{p+q+1}) \cap F^p C^{p+q}}{d(F^p C^{p+q-1} ) + F^{p+1} C^{p+q} } ,
\end{equation}
with differentials $d_0$, $d_1$ induced by $d$ on the quotients.
\end{proposition}

We quickly review the construction of the spectral sequence associated with a filtered graded complex. 
Recall that each homogeneous component $C^p$ has a decreasing filtration $F^i C^p$ and that the stability of the differential implies $d(F^i C^p) \subset F^i C^{p+1}$.
One starts by defining the subspaces of $F^p C^{p+q}$
\begin{align}
&Z_r^{p,q} = d^{-1}( F^{p+r} C^{p+q+1} )\cap F^p C^{p+q}, \\
&B_r^{p,q} = d( F^{p-r} C^{p+q-1} )\cap F^p C^{p+q} 
\end{align}
which form a tower 
\begin{equation}
d( F^{p} C^{p+q-1} )=B_0^{p,q} \subset B_1^{p,q}  \subset \cdots \subset Z_1^{p,q} \subset Z_0^{p,q} = F^p C^{p+q} .
\end{equation}
Notice moreover that $B_{-1}^{p,q} = d( F^{p+1} C^{p+q-1} )$ and $Z_{-1}^{p,q} = F^p C^{p+q} $. 

We define the $r$-th page of the spectral sequence as
\begin{equation} \label{defE}
E_r^{p,q} = \frac{Z_r^{p,q}}{B_{r-1}^{p,q} + Z_{r-1}^{p+1,q-1}} ,
\end{equation}
which is a well-defined quotient since $Z_{r-1}^{p+1, q-1} \subset Z_r^{p,q}$. For $r=0,1$ this gives~\eqref{E0} and~\eqref{E1}.

One can check that
\begin{equation}
d (Z_r^{p,q}) = B_r^{p+r,q-r+1} \subset Z_r^{p+r,q-r+1} ,
\end{equation}
and, consequently, 
\begin{equation}
d( B_{r-1}^{p,q} + Z_{r-1}^{p+1,q-1} ) = B_{r-1}^{p+r,q-r+1}. 
\end{equation}
That means that the differential restricts to a map $Z_r^{p,q} \to Z_r^{p+r,q-r+1}$ which descends to the quotient in~\eqref{defE}, defining the differential $d_r$. Finally one can prove that the cohomology of the differential $d_r$ on $E_r^{p,q}$ coincides with $E_{r+1}^{p,q}$.

\begin{remark}
One can check that $H^{p,q}(E_0^{*,*}, d_0) = E_1^{p,q}$ simply by computing kernel and image of $d_0$ in $E_0^{p,q}$ which shows that the cohomology is
\begin{equation}
\frac{Z_1^{p,q} / F^{p+1} C^{p+q} }{B_0^{p,q} / F^{p+1} C^{p+q}} \cong 
\frac{Z_1^{p,q }}{B_0^{p,q} + F^{p+1} C^{p+q}} = E_1^{p,q}.
\end{equation}
\end{remark}

A filtration $F^* C$ of a graded complex $(C,d)$ is {\it bounded} if for each degree $p$ there are integers $s$ and $t$ such that 
\begin{equation}
0 = F^s C^p \subset \cdots \subset F^{i+1} C^p \subset F^i C^p \subset \cdots \subset F^t C^p = C^p .
\end{equation}

The cohomology  of a filtered graded complex $(C,d)$ inherits a filtration, where $F^i H(C,d)$ is given by the image of $H( F^i C, d)$ in $H(C,d)$ under the inclusion map.

The standard {\it convergence theorem} for the spectral sequence associated with a filtered graded complex states that, under the assumption of boundedness
of the filtration, the limit term of the spectral sequence determines the graded space associated with the filtration of the cohomology of the complex.

More precisely we have, see e.g. \cite{m01}:
\begin{theorem}  \label{spectralthm}
The spectral sequence associated with a bounded filtration converges to $H(C,d)$, i.e.
\begin{equation}
E_\infty^{p,q} \cong \frac{F^p H^{p+q} (C,d)}{F^{p+1} H^{p+q} ( C,d)} .
\end{equation}
\end{theorem}

\section{Bihamiltonian cohomology}

In this Section we prove our main result, i.e. we compute the bihamiltonian cohomology $BH^p_d(\hcF,d_1,d_2)$ of the dispersionless KdV Poisson pencil.

The main step in our proof is the computation of the cohomology of the auxiliary complex $(\hcA[\lambda],D_\lambda)$ in Proposition~\ref{mainproposition}. We define a filtration of such complex by imposing a bound on the highest derivative and we study the convergence of the associated spectral sequence. The first page $E^{*,*}_1$ is still highly non-trivial: we compute it in Lemma~\ref{lemma7} and we give an explicit expression for the differential $d_1$ in Lemma~\ref{lemma9}. The spectral sequence however collapses at the second term. This is proved by introducing an explicit contracting homotopy $h$ for $d_1$. Finally, the convergence theorem for bounded spectral sequences allows us to recover the cohomology of 
$(\hcA[\lambda],D_\lambda)$ by a standard argument. 

In Section~\ref{conseq} we use two arguments from~\cite{lz13} to compute the cohomology groups $H(\hcF[\lambda], d_\lambda)$ and the bihamiltonian cohomology groups $BH(\hcA, D_1,D_2)$ and $BH(\hcF, d_1,d_2)$. The first relates the bihamiltonian cohomology groups $BH(\hcA, D_1, D_2)$, resp. $BH(\hcF, d_1, d_2)$, to the cohomology of the complexes $(\hcA[\lambda], D_\lambda)$ , resp. $(\hcF[\lambda],d_\lambda)$, cf. Lemma~\ref{lamcohomlemma}.
The second employs a long exact sequence~\eqref{long} to relate the cohomology of $(\hcF[\lambda],d_\lambda)$ to that of $(\hcA[\lambda],D_\lambda)$.

%
%

\subsection{The main proposition}

We start by computing the cohomology of the complex $(\hcA[\lambda],D_\lambda)$, where $D_\lambda = D_2 - \lambda D_1$.
\begin{proposition} \label{mainproposition}
\begin{equation} \label{lamcohom}
H(\hcA[\lambda],D_\lambda) = \R[\lambda]\oplus (C^\infty (\R)/ \R[u]  )\theta\theta^1 \oplus C^\infty (\R) \theta\theta^1\theta^2
\end{equation}

\end{proposition}
\begin{proof}
The proof consists of two parts: first we define a filtration of the complex $(\hcA[\lambda] , D_\lambda)$ and compute the first page $E_1^{*,*}$ of the associated spectral sequence. Then we compute the second page $E_2^{*,*}$ by defining an explicit contracting homotopy $h$ for $d_1$. It turns out that the spectral sequence $E_r^{*,*}$ collapses at the second page. Use of the standard convergence theorem for spectral sequences yields the desired result. 

Let 
\begin{equation}
\hcA^{(r)} : = C^\infty (\R) [[u^1, \dots, u^r; \theta, \dots , \theta^r ]],
\end{equation}
the subspace of $\hcA$ where the order of $x$-derivatives is at most $r$. 
The quotient $\hcA^{(r)} / \hcA^{(r-1)}$ is canonically identified with the subspace $\hcA^{[r]}$ generated by monomials in $\hcA^{(r)}$ with at least one $x$-derivative of order $r$. Denote 
\begin{equation}
\hcA_d^{(r)} = \hcA_d \cap \hcA^{(r)}, \qquad
\hcA_d^{[r]}  = \hcA_d \cap \hcA^{[r]} \cong \hcA_d^{(r)} / \hcA_d^{(r-1)}
\end{equation}
the standard degree $d$ homogeneous component of $\hcA^{(r)}$ and $\hcA^{[r]}$, respectively.

We define the following decreasing filtration of $\hcA[\lambda]$
\begin{equation}
F^i \hcA[\lambda] = \left\{ a = \sum_{d\geq 0} a_d \in \hcA[\lambda] \text{ s.t. } a_d \in \hcA_d^{(d-i)}[\lambda] \right\}, 
\end{equation} 
for $i\in\Z$, which induces the filtration of the homogeneous component $\hcA_d[\lambda]$
\begin{equation}
F^i \hcA_d[\lambda] = \hcA_d^{(d-i)}[\lambda] .
\end{equation}
This filtration is bounded i.e.
\begin{equation}
0 = F^{d+1} \hcA_d[\lambda] \subset \cdots \subset F^{i+1} \hcA_d [\lambda] \subset F^i \hcA_d[\lambda] \subset \cdots \subset F^0 \hcA_d [\lambda] = \hcA_d[\lambda] .
\end{equation}

Let $(E_r^{*,*},d_r)$ be the spectral sequence associated with the filtration $F^i \hcA[\lambda]$. By construction the zeroth page $E_0^{*,*}$ is given by the graded space associated with the filtration
\begin{equation}
E_0^{pq} = gr^p \hcA_{p+q}[\lambda] \cong \hcA_{p+q}^{[q]}[\lambda] .
\end{equation}
Note that $E_r^{*,*}$ is a first quadrant spectral sequence, i.e. $E_r^{p,q}$ is non-trivial only if $p,q\geq0$. 

In particular we have
\begin{equation}
E_0^{0,0} =\hcA_0[\lambda], \quad
E_0^{p,0} = 0 \text{ for } p>0
\end{equation}
and 
\begin{equation}
E_0^{0,q} = u^q \hcA_0[\lambda] + \theta^q \hcA_0[\lambda] 
\text{ for } q >0.
\end{equation}

The differential $d_0:E_0^{p,q}\to E_0^{p,q+1} $ is simply the differential induced by $D_\lambda$ on the graded space, i.e.
\begin{equation}
d_0 : \hcA^{[q]}_{p+q}[\lambda] \to \hcA^{[q+1]}_{p+q+1}[\lambda]
\end{equation}
explicitly given by
\begin{equation}
d_0 = \left( (u-\lambda) \theta^{q+1} +\frac12 u^{q+1} \theta \right) \frac{\partial }{\partial u^q} + \frac12 \theta \theta^{q+1} \frac{\partial }{\partial \theta^q} .  
\end{equation}
 
By computing the cohomology of the complex $(E_0^{*,*},d_0)$ we get the first page of the spectral sequence.
\begin{lemma} \label{lemma7}
The non-trivial entries $E_1^{p,q}$ of the first page of the spectral sequence are given by
\begin{align}
& E_1^{0,0} = \R [\lambda], \\
& E_1^{0,1} = \frac{C^\infty(\R)}{\R[u]} \theta \theta^1, \\
& E_1^{p,q} = \hcA_p^{[q-1]} \theta \theta^q 
\text{ for } p\geq1, q\geq2 . \label{bigcase}
\end{align}

\end{lemma}
\begin{proof}
Let us prove~\eqref{bigcase}. Let $p\geq1$, $q\geq1$. The kernel of $d_0$ in $E_0^{p,q}$ is given by 
\begin{equation} \label{kern}
\theta\theta^q \hcA^{(q-1)}_{p}[\lambda] + \left( (u-\lambda) \theta^q + \frac12 u^q \theta \right) \hcA^{(q-1)}_{p} [\lambda].
\end{equation}
To see this, let $f \in \hcA^{[q]}_{p+q} [\lambda]$ and write it as $f = f_0 + \theta f_1$ with both $f_0$ and $f_1$ independent of $\theta$. The equation $d_0 f = 0$ is equivalent to
\begin{align}
& \theta \frac{\partial f}{\partial u^q} = 0 , \\
& (u-\lambda) \frac{\partial f}{\partial u^q} -\frac12 \theta \frac{\partial f}{\partial \theta^q} =0 .   
\end{align}
The first one implies that $f_0$ does not depend on $u^q$ hence it can be written as
\begin{equation}
f_0 = \theta^q g
\end{equation}
for $g \in \hcA^{(q-1)}_{p} [\lambda]$ independent of $\theta$.

The second equation then becomes
\begin{equation}
(u-\lambda) \frac{\partial f_1}{\partial u^q} = \frac12 g .
\end{equation}
This implies that, since $g$ does not depend on $u^q$ and $\theta^q$, then 
\begin{equation}
f_1 = \theta^q f_2 + u^q f_3
\end{equation}
for $f_2, f_3\in \hcA^{(q-1)}_{p}[\lambda]$. Then $f$ is clearly of the form~\eqref{kern}.

Note that from~\eqref{kern} it follows immediately that $E_1^{p,1}=0$ for $p\geq 1$.

Let us now restrict to $q\geq2$. The image of $d_0$ in $E_0^{p,q}$ is given by elements of the form
\begin{equation} \label{imel}
\left( (u-\lambda) \theta^q + \frac12 u^q \theta \right) \frac{\partial \tilde{g}}{\partial u^{q-1}} 
+ \theta^q \theta \left( (u-\lambda) \frac{\partial \tilde{f}}{\partial u^{q-1}} - \frac12 \frac{\partial \tilde{g}}{\partial \theta^{q-1}}  \right) 
\end{equation}
where $\tilde{g} + \theta \tilde{f}$ is an arbitrary element in $\hcA^{[q-1]}_{p+q-1}[\lambda]$ and $\tilde{f}$ and $\tilde{g}$ do not depend on $\theta$. 

Let $f$ be an element of the kernel~\eqref{kern}, i.e. of the form
\begin{equation}
f= \left( (u-\lambda) \theta^q + \frac12 u^q \theta \right) g + \frac12 \theta \theta^q h .
\end{equation} 
Note that we can choose both $g$ and $h$ independent of $\theta$.

Let us now find another representative for the class $[f]$ in $E_1^{p,q}$ by subtracting an element of the image of $d_0$ of the form~\eqref{imel}.
Choosing $\tilde{g}$ such that
\begin{equation} \label{47}
g = \frac{\partial \tilde{g}}{\partial u^{q-1}}
\end{equation}
we obtain a representative of the form
\begin{equation} \label{48}
\frac12 \theta\theta^q \left( h
- \frac{\partial \tilde{g}}{\partial \theta^{q-1} }  
+2 (u-\lambda) \frac{\partial \tilde{f}}{\partial u^{q-1}}  \right) .
\end{equation}

Since $\tilde{g}$ is a polynomial in $u^{q-1}$, eq.~\eqref{47} fixes completely $\tilde{g}$ but for its part of degree zero in $u^{q-1}$. Let $\tilde{g} = \tilde{g}_0 + \tilde{g}_1$ where $\tilde{g}_0$ is the degree-zero part in $u^{q-1}$. 
Since $\frac{\partial \tilde{g}_0}{\partial \theta^{q-1}}$ is an arbitrary term in $\hcA^{(q-2)}_{p}[\lambda]$ which does not depend on $\theta$, we can choose $\tilde{g}_0$ such that
\begin{equation}
h + \frac{\partial \tilde{g}_1}{\partial \theta^{q-1}} + \frac{\partial \tilde{g}_0}{\partial \theta^{q-1}}   
\end{equation} 
is an element of $\hcA^{[q-1]}_{p}[\lambda]$.
Finally we can choose $\tilde{f}$ to kill the $\lambda$ dependence in~\eqref{48}. Eq.~\eqref{bigcase} is proved.

It remains to show that $E_1^{0,q} =0$ for $q\geq2$, $E_1^{0,1} = C^\infty (\R) \slash \R[u]$ and $E_1^{0,0}=\R[\lambda]$, which are straightforward computations.
\end{proof}

\begin{remark}
Note that $E_1^{p,q} =0$ for $p \leq q-2$, $q \geq2$.
\end{remark}

\begin{lemma} \label{lemma9}
For $p\geq1$, $q\geq2$, the differential $d_1: E_1^{p,q} \to E_1^{p+1,q}$ is given by
\begin{equation} \label{d1}
d_1 (f \theta \theta^q) = \left( \left(D_{\lambda} (f) \right)_{\lambda=u} + \frac{q-2}{2} \theta^1 f \right) \theta \theta^q   \in \frac{\hcA^{(q-1)}_{p+1} \theta \theta^q }{\hcA^{(q-2)}_{p+1} \theta\theta^q}
\end{equation}
for $f \in \hcA_p^{[q-1]}$.
\end{lemma}
\begin{proof}
By definition the differential $d_1$ is the map induced by $D_\lambda$ on the first page of the spectral sequence written as
\begin{equation} \label{53}
E_1^{p,q} = \frac{D_\lambda^{-1} \hcA^{(q)}_{p+q+1}[\lambda] \cap \hcA^{(q)}_{p+q}[\lambda] }{D_\lambda \hcA^{(q-1)}_{p+q-1}[\lambda] + \hcA^{(q-1)}_{p+q} [\lambda] } .
\end{equation}
For $p\geq1$, $q\geq2$, we can choose the representative of a class in  $E_1^{p,q}$ as an element in $\hcA^{(q)}_{p+q}[\lambda]$  of the form $f \theta\theta^q$ with $f \in \hcA_{p}^{[q-1]}$. Applying $D_\lambda$ to such element, and quotienting by $\hcA^{(q-1)}_{p+q+1}$, we get
\begin{equation}  \label{54}
d_1 (f \theta \theta^q) = D_\lambda (f) \theta \theta^q + \frac{q-2}{2} \theta^1 f \theta \theta^q  ,
\end{equation}
where the right-hand represents an element in  
\begin{equation} \label{quot1}
 \frac{ \hcA^{[q]}_{p+q+1}[\lambda] }{D_\lambda \hcA^{(q-1)}_{p+q}[\lambda] \cap \hcA^{[q]}_{p+q+1} [\lambda] } .
\end{equation}
Next, we quotient by $D_\lambda \hcA^{(q-1)}_{p+q}[\lambda] \cap \hcA^{[q]}_{p+q+1} [\lambda]$ to get a representative of $d_1(f\theta\theta^q)$ in $\hcA^{(q-1)}_{p+1}\theta\theta^q$. 

In other words we can add to~\eqref{54} an element of the form 
\begin{equation}
( u \theta^q + \frac12 u^q \theta) \frac{\partial g}{\partial u^{q-1}} + \frac12 \theta \theta^q \frac{\partial g}{\partial \theta^{q-1}} 
-\lambda \theta^q \frac{\partial g}{\partial u^{q-1}}   
\end{equation}
for $g \in \hcA^{(q-1)}_{p+q}[\lambda]$ such that it kills the $\lambda$ term in~\eqref{54} i.e.
\begin{equation} \label{57}
\theta^q \left( \frac{\partial g}{\partial u^{q-1}} - \theta D_1(f) \right) = 0. 
\end{equation}
It is clear that we can always choose an element $g = \theta \tilde{g}$ such that this equation holds. We obtain a new representative for~\eqref{54}, that is
\begin{equation}
\left( D_2(f) + \frac{q-2}2 \theta^1 f - u \frac{\partial \tilde{g}}{\partial u^{q-1}}  \right) \theta\theta^q ,
\end{equation}
which, using~\eqref{57}, coincides with the desired result~\eqref{d1}.

It remains to show that this element lives in the quotient
\begin{equation} \label{quot2}
\frac{\hcA^{(q-1)}_{p+1} \theta \theta^q }{\hcA^{(q-2)}_{p+1} \theta\theta^q} . 
\end{equation}
Let us look for elements in the denominator of~\eqref{quot1} which are of the form $h \theta\theta^q$ with no $\lambda$ dependence. An element of $D_\lambda \hcA^{(q-1)}_{p+q}[\lambda] \cap \hcA^{[q]}_{p+q+1} [\lambda]$ is of the form
\begin{equation}
\left( (u-\lambda) \theta^q + \frac12 u^q \theta\right) \frac{\partial g}{\partial u^{q-1}} +\frac12 \theta\theta^q \frac{\partial g}{\partial \theta^{q-1}}  
\end{equation}
for $g\in \hcA^{(q-1)}_{p+q} [\lambda]$. To get an element of the form $h \theta\theta^q$ either $\frac{\partial g}{\partial u^{q-1}}$ is proportional to $\theta$, but then we would obtain an element dependent on $\lambda$, or $\frac{\partial g}{\partial u^{q-1}} =0$ with $g\in \hcA^{(q-1)}_{p+q}$.  Without loss of generality we can then choose $g \in \theta^{q-1} \hcA^{(q-2)}_{p+1}$. We conclude that~\eqref{d1} lives indeed in~\eqref{quot2}.
\end{proof}

Now we compute the second page $E_2^{*,*}$. 

\begin{lemma} \label{lemma10}
The second page of the spectral sequence is given by
\begin{equation}
 E_2^{p,q} = \begin{cases}
\R[\lambda] & p=0, q=0, \\
\frac{C^\infty (\R)}{\R[u]} \theta \theta^1 & p=0, q=1 \\
C^\infty (\R) \theta \theta^1 \theta^2 & p=1, q=2 \\
0 &\text{else.}
\end{cases}
\end{equation}
\end{lemma}
\begin{proof}
Let us first compute the nontrivial entries. By the vanishing of $d_1$, the spectral sequences stabilizes at $E_1^{0,0}$ and $E_1^{0,1}$, hence $E_2^{0,0} =\R[\lambda]$ and $E_2^{0,1}  = \frac{C^\infty (\R)}{\R[u]} \theta\theta^1$ . On the other hand, an element $(f(u) \theta^1 + g(u) u^1 ) \theta\theta^2$ of $E_1^{1,2} = \hcA_1^{[1]} \theta \theta^2$ is in the kernel of $d_1$ iff
\begin{equation}
d_1(f \theta^1 + g u^1 ) = \frac12 g u^1 \theta \theta^1 \theta^2 =0 .
\end{equation}
Hence
\begin{equation}
E_2^{1,2} = \ker ( d_1: E_1^{1,2} \to E_1^{2,2} ) = C^\infty (\R) \theta \theta^1 \theta^2 .
\end{equation}
For the cases $p\geq 1$ and $q\geq 2$, $(p,q)\not = (1,2)$, we shall construct an explicit homotopy for the differential 
$d_1$. Introduce the operator
\[
U:=\sum_{s\geq 1}\frac{(s+2)}{2}u^s\frac{\partial}{\partial u^s}+\sum_{s\geq 0}\frac{(s-1)}{2}\theta^s
\frac{\partial}{\partial \theta^s}.
\]
Remark that this operator measures the degree of a monomial in $\hat{\mathcal{A}}$ where 
the weight of each $u^s$ is $(s+2)/2$ and of $\theta^s$ is $(s-1)/2$. 
From this it follows that $U$ is invertible on $E_2^{p,q}=\hat{\mathcal{A}}^{[q-1]}_p\theta\theta^q$ for 
$p\geq 1,~q\geq 2$ and $(p,q)\not = (1,2)$. 
\begin{lemma}
The operator 
\[
h:=U^{-1}\frac{\partial}{\partial\theta^1}
\]
is a contracting homotopy for $d_1:E^{p,q}_1\to E_1^{p+1,q}$ for $p\geq 1,~q\geq 2$ and $(p,q)\not = 
(1,2)$.
\end{lemma}
\begin{proof}
We should check that $hd_1+d_1h=1$.
For this we split the differential $d_1$ as 
\[
d_1=U\theta^1+\alpha.
\]
Since $U$ and $\theta^1$ clearly commute, the first term on the right hand side is a differential
for which $h$ is clearly a contracting homotopy, i.e., $h (U\theta^1)+(U\theta^1)h=1$.

Comparing with the explicit formula of the differential $d^1$ in Lemma \ref{lemma9}, we see that the 
first term $U\theta^1$ in the expression for $d^1$ above is exactly the part that depends on $\theta^1$.
(The last term in Lemma \ref{lemma9} proportional to $(q-2)\theta^1/2$ is taken into account by the 
$s=0$ term in the second summation in the formula for $U$.) Therefore $\alpha$ is exactly the part 
which does not contain $\theta^1$. By the fact that the number of $\theta$'s
in each summand for $d^1$ is odd, it follows that 
\[
\frac{\partial}{\partial\theta^1}\alpha+\alpha\frac{\partial}{\partial\theta^1}=0.
\]
Next, again from the explicit expression of Lemma \ref{lemma9}, we easily check that $\deg(\alpha)=0$, 
for the grading as described above, i.e., where $\deg(u^s)=(s+2)/2$ and $\deg(\theta^s)=(s-1)/2$. In 
other words, $[U,\alpha]=0$, and therefore
\[
h\alpha+\alpha h=0.
\]
Combining this equation with the fact that $h$ is already a contracting homotopy for $U\theta^1$,
we get that $hd^1+d^1 h=1$ and the statement of the Lemma follows.
\end{proof}

%
%
%
%
%

We conclude that $E_2^{p,q}=0$ for $p\geq1$, $q\geq2$, excluding the case $p=1$, $q=2$. This proves Lemma~\ref{lemma10}.
\end{proof}

Let us return to the proof of Proposition~\ref{mainproposition}. 
Since $d_2$ maps $E_2^{p,q}$ to $E_2^{p+2,q-1}$, it must vanish everywhere, in other words the spectral sequence collapses at the second term, and the limit term $E_\infty^{*,*}$ coincides with $E_2^{*,*}$.

The filtration $F^*$ of the complex $(\hcA[\lambda],D_\lambda)$ is bounded, hence it converges by Theorem~\ref{spectralthm}, i.e.\begin{equation}
E_\infty^{p,q} \cong \frac{F^p H_{p+q} ( \hcA[\lambda], D_\lambda )}{F^{p+1} H_{p+q} (\hcA[\lambda], D_\lambda )} .
\end{equation}
It follows in particular that for every $(p,q)$ such that $E_\infty^{p,q}=0$ we have
\begin{equation}
F^p H_{n} ( \hcA[\lambda], D_\lambda) \cong 
F^{p+1}H_{n} ( \hcA[\lambda], D_\lambda)
\end{equation}
where $n=p+q$, and because the filtration is bounded we have 
\begin{equation}
F^0 H_{n}  ( \hcA[\lambda], D_\lambda) = H_{n}  ( \hcA[\lambda], D_\lambda)
\end{equation}
and
\begin{equation}
F^n H_{n}  ( \hcA[\lambda], D_\lambda) = 0 \text{ for } n\not= 0.
\end{equation}

It is straightforward to conclude that the cohomology is given by~\eqref{lamcohom}. The Proposition is proved. 
\end{proof}

\subsection{Consequences} \label{conseq}
The following Lemma, proved in~\cite{lz13}, relates the bihamiltonian cohomology of $\hcA$, resp. $\hcF$, with the cohomology of the complex $(\hcA[\lambda],D_\lambda)$, resp. $(\hcF[\lambda],d_\lambda)$. We give a slightly modified version adapted to the scalar case.
\begin{lemma} \label{lamcohomlemma}
We have
\begin{align} 
& BH_d^p(\hcA, D_1, D_2) \cong H_d^p(\hcA[\lambda], D_\lambda ), \label{bha} \\
& BH_d^p(\hcF, d_1, d_2) \cong H_d^p(\hcF[\lambda], d_\lambda ) \label{bhf}
\end{align}
for $p,d \geq0$ excluding $(p,d) = (0,0), (1,0), (1,1), (2,1)$.
\end{lemma}
\begin{proof}
One can repeat the proof given in~\cite{lz13}. Such proof of~\eqref{bha} relies on the fact that $H^p_d(\hcA, D_1) = 0$ and $H^{p-1}_{d-1}(\hcA, D_1) = 0$. Since in the scalar case $H^p_d(\hcA, D_1)=0$ unless $(p,d) = (0,0), (1,0)$, we need to exclude the four cases above. The analogous remark applies to~\eqref{bhf}.
\end{proof}

\begin{theorem}
The bihamiltonian cohomology groups are given by
\begin{equation}
BH^p_d (\hcA, D_1, D_2) = \begin{cases}
\R & p=0, d=0, \\
C^\infty (\R) \theta \theta^1 &p=2, d=1, \\
C^\infty(\R) \theta \theta^1 \theta^2 &p=3, d=3, \\
0 & else.
\end{cases}
\end{equation}
\end{theorem}
\begin{proof}
It is a direct consequence of~\eqref{lamcohom} and~\eqref{bha} for all $p,d\geq 0$ but the four cases $(p,d) = (0,0), (1,0), (1,1), (2,1)$. The four exceptional cases can be quickly obtained by straightforward calculations.
\end{proof}

\begin{remark}
The bihamiltonian cohomology groups $BH^p_d(\hcA)$ for $p=2,3,4$, $d\geq0$ were computed in \cite{lz13} by direct computation, i.e. without the use of Lemma~\ref{lamcohomlemma}. 
\end{remark}

As noted in~\cite{lz13} the short exact sequence of differential complexes
\begin{equation}
0 \to ( \hcA[\lambda] \slash \R[\lambda], D_\lambda) \xrightarrow{\partial}  ( \hcA[\lambda], D_\lambda ) \xrightarrow{\int}  ( \hcF[\lambda], d_\lambda ) \to 0 
\end{equation}
induces a long exact sequence in cohomology
\begin{multline} \label{long}
\cdots \to H_{d-1}^p(\hcA[\lambda]\slash \R[\lambda] ) \to H^p_d (\hcA[\lambda] ) \to \\\to H^p_d (\hcF[\lambda]) \to  H^{p+1}_d (\hcA[\lambda] \slash \R[\lambda]) \to \cdots
\end{multline}
for $d\geq1$. Note that $H^{p}_d (\hcA[\lambda] \slash \R[\lambda], D_\lambda) = H^{p}_d (\hcA[\lambda], D_\lambda)$ unless $p=d=0$.
The vanishing of most of the groups $H_d^p(\hcA[\lambda],D_\lambda)$ implies at once the vanishing of $H_d^p(\hcF[\lambda], d_\lambda)$ for $p\geq1$ excluding the four cases $(p,d) = (1,1)$, $(2,1)$, $(2,3)$, $(3,3)$.
Moreover the long exact sequence~\eqref{long} at $(p,d)=(1,1)$, $(2,1)$ implies that
\begin{equation}
H_1^1(\hcF[\lambda]) \cong H_1^2(\hcF[\lambda]) = H_1^2(\hcA[\lambda]) = ( C^\infty (\R) \slash \R[u] )  \theta\theta^1
\end{equation}
and at $(p,d) = (2,3), (3,3)$ that
\begin{equation}
H_3^2(\hcF[\lambda]) \cong H_3^3(\hcF[\lambda]) = H_3^3(\hcA[\lambda]) = C^\infty (\R)  \theta \theta^1\theta^2 .
\end{equation}
It is easy to compute the $d=0$ cases explicitly, thus obtaining the following:
\begin{proposition} \label{cohF}
The non-trivial cohomology groups $H_d^p (\hcF[\lambda], d_\lambda )$ are isomorphic to
\begin{equation}
C^\infty (\R) \slash \R[u] \text{ for } (p,d)=(1,1), (2,1),
\end{equation}
\begin{equation}
C^\infty (\R) \text{ for } (p,d)=(2,3), (3,3) 
\end{equation}
and $\R[\lambda]$ for $(p,d)=(0,0)$. 
\end{proposition}

\begin{theorem}
For the bihamiltonian cohomology groups $BH_d^p(\hcF,d_1,d_2)$ we have
\begin{equation}
 BH_d^p(\hcF,d_1,d_2) \cong \begin{cases}
 C^\infty (\R) & (p,d) = (1,1), (2,1), (2,3), (3,3), \\
 \R & (p,d)=(0,0), \\
 0 & else.
 \end{cases}
\end{equation}
\end{theorem}
\begin{proof}
Follows from Lemma~\ref{lamcohomlemma} and Proposition~\ref{cohF}. 

The four exceptional cases that cannot be obtained from
Lemma~\ref{lamcohomlemma} have to be checked separately. It is a simple
straightforward computation, so we provide here only an example for
$p=d=1$. The bihamiltonian cohomology group in this case is given by elements of $\hcF_1^1$ that are in the kernel of both $d_1$ and $d_2$. A general element of $\hcF_1^1$ is of the form 
\begin{equation}
\int \left( f(u) \theta^1 + g(u) u^1 \theta \right) .
\end{equation}
This element is in the kernel of both $d_1$ and $d_2$, since, as one can easily check, both $D_1$ and $D_2$ map $f(u) \theta^1 + g(u) u^1 \theta$ to $\partial \hcA_1^2$. (Recall that $d_i \int = \int D_i$ on $\hcA$). Finally we can quotient by $\partial \hcA_0^1$ to obtain a representative of the form
\begin{equation}
\int  h(u) u^1 \theta .
\end{equation}
The remaining exceptional cases can be computed in a similar way.
\end{proof}




\subsection*{Acknowledgments}

The authors were supported by the Netherlands Organization for Scientific Research.

\end{document}